\documentclass[a4paper,12pt,reqno]{amsart}

\makeindex

\usepackage{amsfonts,graphics,amsmath,amsthm,amsfonts,amscd,amssymb,amsmath,latexsym,euscript, enumerate}
\usepackage{epsfig,url}
\usepackage{flafter}
\usepackage[all,cmtip,line]{xy}
\usepackage{array}
\usepackage[english]{babel}
\usepackage{overpic}
\usepackage{subfig}
\usepackage{multirow}

\usepackage{wrapfig}
\usepackage[shortlabels]{enumitem}
\setlist[enumerate, 1]{1\textsuperscript{o}}

\usepackage[usenames,dvipsnames,svgnames,table]{xcolor}
\usepackage{graphicx}
\usepackage[bookmarks=true,bookmarksnumbered=true,
 pdftitle={},
 pdfsubject={},
 pdfauthor={},
 pdfkeywords={TeX, dvipdfmx, hyperref, color,},
 colorlinks=true, linkcolor=RoyalBlue, citecolor=Emerald]
 {hyperref}

\newtheorem{theorem}{Theorem}[section]
\newtheorem{lemma}[theorem]{Lemma}
\newtheorem{proposition}[theorem]{Proposition}

\theoremstyle{corollary}

\theoremstyle{definition} % italic or bold etc.
\newtheorem{definition}[theorem]{Definition}
\newtheorem{definition-lemma}[theorem]{Definition-Lemma}

\newtheorem{example}[theorem]{Example}

\newtheorem{remark}[theorem]{Remark}

\numberwithin{equation}{section}

\newcommand{\R}{\mathbb{R}}

\newcommand{\Q}{\mathbb{Q}}

\DeclareMathOperator{\ord}{ord}

\DeclareMathOperator{\Mob}{Mob}

\DeclareMathOperator{\eff}{eff}

\DeclareMathOperator{\Div}{Div}

\DeclareMathOperator{\FT}{FT}
\DeclareMathOperator{\pFT}{pFT}

\def\mult{\operatorname{mult}}

\def\Supp{\operatorname{Supp}}
\def\Exc{\operatorname{Exc}}

\def\lct{\operatorname{lct}}
\def\plct{\operatorname{plct}}

\usepackage{mathtools}

\DeclarePairedDelimiterX{\inp}[2]{\langle}{\rangle}{#1, #2}
\DeclarePairedDelimiterX{\norm}[1]{\lVert}{\rVert}{#1}
\usepackage{mathtools}

\title[plc thereshold]
{ACC of plc threshold}

\begin{document}

\author{Sung Rak Choi}
\author{Sungwook Jang}
\address{Department of Mathematics, Yonsei University, 50 Yonsei-ro, Seodaemun-gu, Seoul 03722, Republic of Korea}
\email{sungrakc@yonsei.ac.kr}
%\address{Department of Mathematics, Yonsei University, 50 Yonsei-ro, Seodaemun-gu, Seoul 03722, Republic of Korea}
\email{swjang@yonsei.ac.kr}

%\thanks{}

%\subjclass[2010]{14C20, 14D06.}
\date{\today}

\begin{abstract}
In this paper, we define the potential log canonical threshold and prove the ascending chain condition (ACC) of the set of these thresholds satisfies. We also consider collections of Fano type varieties and study their basic properties including boundedness.
\end{abstract}

\maketitle
%\tableofcontents

%%%%%%%%%%%%%%%%%%%%%%%%%%%%%%%%%%%%%%%%%%%%%%%%%%%%%
\section{Introduction}
%%%%%%%%%%%%%%%%%%%%%%%%%%%%%%%%%%%%%%%%%%%%%%%%%%%%%

Recently, various thresholds in birational geometry have attracted  much attention as their interesting behaviours (e.g., ascending chain condition, ACC) turned out to have important implications especially in the minimal model program. In this paper, we are interested in the so called \textit{potential thresholds} for the pair $(X,\Delta)$ such that $-(K_X+\Delta)$ is pseudoeffective. It is a generalization of the well-known log canonical thresholds.
The notion of potential pairs was first introduced and studied in \cite{CP16}. It was defined as a means to bound the singularities of the outcome of the $-(K_X+\Delta)$-minimal model program (MMP).
We believe that the deeper understanding of this notion will be useful in establishing the  $-(K_X+\Delta)$-MMP.

For a pair $(X,\Delta)$ with an effective divisor $\Delta$ such that $-(K_X+\Delta)$ is pseudoeffective, consider a birational morphism $f:Y\to X$ and the divisorial Zariski decomposition $-f^*(K_X+\Delta)=P+N$. If $F$ is a prime divisor on $Y$, $a(F;X,\Delta)$ denotes the log discrepancy of $(X,\Delta)$ at $F$. We  define the \emph{potential log discrepancy} of $(X,\Delta)$ at $F$ as $\bar{a}(F;X,\Delta):=a(F;X,\Delta)-\mult_FN$.
The given pair $(X,\Delta)$ is said to be potentially klt (resp. $\epsilon$-potentially lc for $\epsilon\geq 0$) if $\inf_{f,F}\bar{a}(F;X,\Delta)>0$ (resp. $\geq\epsilon$) holds. We write pklt for potentially klt and $\epsilon$-plc for $\epsilon$-potentially lc. See Definition \ref{def:plc} and \cite{CP16} for more details.

%\begin{definition}%[potential klt]
%Let $(X,\Delta)$ be a pair such that $-(K_{X}+\Delta)$ is pseudoeffective. Let $f:Y\to X$ be a log resolution for $(X,\Delta)$ and let $F$ be a prime divisor on $Y$. Write $K_{Y}+\Delta_{Y}=f^{\ast}(K_{X}+\Delta)$.
%\begin{enumerate}[(1)]
%\item The \textit{log discrepancy} $a(F;X,\Delta)$ of $(X,\Delta)$ at $F$ is defined as $1-\mult_{F}\Delta_{Y}$ and the \textit{potential log discrepancy} $\bar{a}(F;X,\Delta)$ of $(X,\Delta)$ at $F$ is defined as
%\[\bar{a}(F;X,\Delta):=a(F;X,\Delta)-\mult_{F}N,\]
%where $N$ is the negative part of divisorial Zariski decomposition of $-f^{\ast}(K_{X}+\Delta)$.

%\item We say that the pair $(X,\Delta)$ is \textit{potentially kawamata log terminal} (\textit{pklt} for short) if $\inf_{F}\bar{a}(F;X,\Delta)>0$ holds, where the infimum is taken over all log resolutions $f:Y\to X$ of $(X,\Delta)$ and all prime divisors $F$ on $Y$. The pair $(X,\Delta)$ is said to have \textit{log canonical singularities} if $\inf_{F}\bar{a}(F;X,\Delta)\geq 0$ holds  in the above expression. For a real number $\epsilon\geq 0$, we say that $(X,\Delta)$ is \textit{$\epsilon$-potentially log canonical}) (\textit{$\epsilon$-plc} if $\inf\bar{a}(F;X,\Delta)\ge \epsilon$ holds in the above expression.
%\end{enumerate}
%\end{definition}

By construction, the notions of pklt, $\epsilon$-plc capture the singularities of the pair $(X,\Delta)$ as well as the positivity of $-(K_X+\Delta)$. However, if $-(K_X+\Delta)$ is nef, then the notions of pklt, $\epsilon$-plc coincide with the usual klt, $\epsilon$-lc. One of the most important properties of these notions is their invariant property under $-(K_X+\Delta)$-MMP (Proposition \ref{prop:MMP invariant}). %See \cite{CP16} for more details.

Recall that the log canonical threshold $\lct(X,\Delta;D)$ of a klt pair $(X,\Delta)$ with respect to a given divisor $D$ is defined as the supremum of the real numbers $t$ for which $(X,\Delta+tD)$ is lc. The \emph{potential lct} $\plct(X,\Delta;D)$ of a plc pair $(X,\Delta)$ with respect to $D$ is defined similarly and by definition it is at most the effective threshold $t_{\eff}(X,\Delta;D):=\sup\{t\geq 0| -(K_X+\Delta+tD) \text{ is pseudoeffective}\}$. See Definition \ref{def:plc} for the precise definition.

%We prove that the set of $\plct(X,\Delta;D)$ in a fixed dimension satisfies the similar properties of the set of $\lct(X,\Delta;D)$.

The following are the main results of this paper.
We first prove the ACC of effective thresholds $t_{\eff}(X,\Delta;D)$ for $\epsilon$-plc pairs $(X,\Delta)$ with $\epsilon>0$ in a fixed dimension.

\begin{theorem}\label{thrm:main t_eff}
Let $I$ be a DCC set of nonnegative real numbers, $n$ a natural number and $\epsilon$ a positive real number. Then there exists an ACC set $I'$ such that the pseudoeffective threshold $t_{\eff}=t_{\eff}(X,\Delta;D)$ belongs to $I'$, whenever the pair $(X,\Delta)$ and the divisor $D$ satisfy the following conditions:
\begin{enumerate}[$(1)$]%[label=\upshape(\roman*)]
\item $(X,\Delta)$ is an $\epsilon$-plc pair with $X$ of dimension $n$,
\item $-(K_{X}+\Delta)$ is big,
\item $D$ is a nef divisor, and
\item the coefficients of $\Delta$ and $D$ belong to $I$.
\end{enumerate}
\end{theorem}

By Example \ref{eg:e=0}, we cannot let $\epsilon=0$ in Theorem \ref{thrm:main t_eff}.

%By definition, there is a possibility that the $\plct(X,\Delta;D)$ appears as the effective threshold $t_{\eff}=t_{\eff}(X,\Delta;D)$ for which $(X,\Delta+t_{\eff}D)$ is still pklt. However, by Proposition \ref{prop:plct=rclct}, this is not the case at least for Fano type varieties.

%\begin{theorem}
%Let $I$ be a DCC set of nonnegative real numbers and $n$ a natural number. Then there exists an ACC set $I'$ such that the pseudo effective threshold $t_{\eff}(X,\Delta,D)$ belongs to $I'$, whenever the pair $(X,\Delta)$ and the divisor $D$ satisfy the followings:
%\begin{enumerate}[label=\upshape(\roman*)]
%\item $(X,\Delta)$ is a $\epsilon$-plc pair with $X$ of dimension $n$,
%\item $-(K_{X}+\Delta)$ is big,
%\item $D$ is a nef divisor,
%\item the coefficients of $\Delta$ and $D$ belong to $I$.
%\end{enumerate}
%\end{theorem}

\begin{theorem}\label{thrm:main}
Suppose that $I$ is a DCC set of nonnegative real numbers and $n$ is a natural number. Then there is an ACC set $\Sigma$ of real numbers such that the plc threshold $\plct(X,\Delta;D)$ of $(X,\Delta)$ with respect to $D$ belongs to $\Sigma$ whenever $(X,\Delta)$, $D$ satisfy the following conditions
\begin{enumerate}[$(1)$]
\item $(X,\Delta)$ is a pklt pair of dimension $n$,
\item $-(K_{X}+\Delta)$ is big, and
\item the coefficients of $\Delta$ and $D$ belong to $I$.
\end{enumerate}
\end{theorem}

If $-(K_X+\Delta)$ is nef, then the notion of plc coincides with the usual lc for any $\epsilon\geq 0$. Thus in this case, Theorem \ref{thrm:main} is implied by Theorem 1.1 of \cite{HMX14}. Note that by Theorem \ref{thrm:FT criterion}, the conditions (1),(2) imply that $X$ is of Fano type. In Theorem \ref{thrm:main}, if we assume that $X$ is of Fano type, then we can relax the conditions (1),(2) in Theorem \ref{thrm:main} as follows
\begin{enumerate}[$(1)'$]%[label=\upshape(\roman*)']
\item $(X,\Delta)$ is a plc pair of dimension $n$,
\item $-(K_{X}+\Delta)$ is pseudoeffective.
\end{enumerate}
Then the result is also implied by \cite[Theorem 21]{sho} and \cite[Theorem 8.20]{HLS} due to Proposition \ref{prop:plct=rclct}. However, our proof for Theorem \ref{thrm:main} is independent and relies only on some results of pklt, not of $\R$-complements.

%if we could let $\epsilon=0$ in Theorem \ref{thrm:main}, then it coincides with the main result Theorem 1.1 of \cite{HMX14}. However, since Theorem \ref{thrm-t_eff} is a special case of Theorem \ref{thrm:main} and Theorem \ref{thrm-t_eff} does not hold for $\epsilon=0$, Theorem \ref{thrm:main} does not hold for $\epsilon=0$ in general either.

We also prove that for a fixed positive integer $n$ and a positive number $\epsilon$,  the collection of $\epsilon$-plc pairs $(X,\Delta)$ such that $-(K_X+\Delta)$ is big is bounded. See Section \ref{sec4}.

\begin{theorem}
Let $d$ be a positive integer and $\epsilon$ a positive real number. Then the set of projective varieties $X$ satisfying the following conditions:
\begin{enumerate}[$(1)$]%[label=\upshape(\roman*)]
\item $X$ is a projective variety of dimension $d$,
\item there is a boundary divisor $\Delta$ such that $(X,\Delta)$ is a $\epsilon$-plc and $-(K_{X}+\Delta)$ is big
\end{enumerate}
is a bounded family.
\end{theorem}

\bigskip

This paper is organized as follows. In Section 2, we give the definition of potential lct and gather various materials for the proof of our main result. In Section 3, we prove  Theorem \ref{thrm:main}. In Section 4, we study exhausting sequences for the set of varieties of Fano type.

\section*{Acknowledgement}
The authors thank professors Jingjun Han, Jihao Liu and V.V. Shokurov for the comments and remarks and also for informing us of the results \cite[Theorem 21]{sho} and \cite[Theorem 8.20]{HLS}.

\section{Preliminaries}

In this section, we collect basic definitions and results that are necessary in the proof of the main results. Throughout the paper, we consider $\mathbb Q$-factorial normal projective varieties defined over some algebraically closed field of characteristic $0$. Any divisor will be an $\R$-divisor unless otherwise stated. Here, $\Q$-factoriality is assumed for convenience in order to consider freely the pull-backs of various divisors. However, this assumption may be avoided by taking $\Q$-factorializations.

\subsection{Positivity}\hfill

Let $\Gamma$ be a prime divisor on a variety $X$. If $D$ is big, then we define the asymptotic divisorial valuation $\sigma_{\Gamma}(D)$ of $\Gamma$ along $D$ as
\[\sigma_{\Gamma}(D):=\inf\{\mult_{\Gamma}\Delta ~|~ \Delta\in |D|_{\mathbb{R}}\},\]
where $|D|_{\mathbb{R}}=\{D'\ge 0 ~|~ D'\sim_{\mathbb{R}} D\}$. If $D$ is only pseudoeffective, then we define $\sigma_{\Gamma}(D):=\lim_{\varepsilon\to 0}\sigma_{\Gamma}(D+\varepsilon A)$ where $A$ is an ample divisor. Here, the asymptotic divisorial valuation $\sigma_{\Gamma}(D)$ is independent of the choice of $A$.
By \cite[Proposition III.1.10]{Nak04}, it is known that there are only finitely many prime divisors $\Gamma$ such that $\sigma_{\Gamma}(D)>0$.

\begin{definition}[divisorial Zariski decomposition] \label{definition_ZD}
For a pseudoeffective divisor $D$ on $X$, we define the \emph{negative part} $N_\sigma(D)$ of $D$ as
\[N_{\sigma}(D):=\sum_{\Gamma}\sigma_{\Gamma}(D)\Gamma\]
where the sum is taken over all prime divisors on $X$, and  the \emph{positive part} of $D$ as $P_{\sigma}(D):=D-N_{\sigma}(D)$. We call the expression $D=P_{\sigma}(D)+N_{\sigma}(D)$ the \emph{divisorial Zariski decomposition} of $D$.
If $P_{\sigma}(D)$ is nef (resp. semiample), then we call $D=P_{\sigma}(D)+N_{\sigma}(D)$ the Zariski decomposition (resp. good Zariski decomposition) of $D$.
\end{definition}

It is known that the positive part $P$ of $D$ carries much of the positivity of $D$.

\subsection{Singularity}\hfill

By definition, a pair $(X,\Delta)$ consists of a $\Q$-factorial normal projective variety $X$ and a boundary $\mathbb{Q}$-divisor $\Delta$ on $X$, i.e., the coefficients of $\Delta$ are real numbers in $[0,1]$.

\begin{definition}
Let $(X,\Delta)$ be a pair and $f:Y\to X$ be a proper birational morphism with $Y$ a normal projective variety. A prime divisor $F$ on such $Y$ is called a prime divisor over $X$. Write $K_{Y}+\Delta_{Y}=f^{\ast}(K_{X}+\Delta)$ for some divisor $\Delta_Y$ on $Y$. Then the \textit{log discrepancy} $a(F;X,\Delta)$ of $(X,\Delta)$ at $F$ is defined as $a(F;X,\Delta):=1-\mult_{F}\Delta_{Y}$. We say that $(X,\Delta)$ is \textit{klt} if $\inf_{F}a(F;X,\Delta)>0$, where the infimum is taken over all prime divisors $F$ over $X$. For a real number $\epsilon\geq 0$, a pair $(X,\Delta)$ is \textit{$\epsilon$-lc} if $\inf_{F}a(F;X,\Delta)\ge \epsilon$ and we just say a pair $(X,\Delta)$ is \textit{lc} if it is 0-lc.
\end{definition}

To determine whether a given pair $(X,\Delta)$ is klt or $\epsilon$-lc, it suffices to compute the log discrepancies for the divisors on a fixed log resolution of $(X,\Delta)$.

\begin{definition} \label{def:plc}
Let $(X,\Delta)$ be a pair such that $-(K_X+\Delta)$ is pseudoeffective and let $F$ be a prime divisor over $X$.
\begin{enumerate}[(1)]
\item The \textit{potential log discrepancy} $\bar{a}(F;X,\Delta)$ of $(X,\Delta)$ at $F$ is defined as
$$ \bar{a}(F;X,\Delta):=a(F;X,\Delta)-\mult_F N, $$
where $F$ is a prime divisor on some model $Y$ with $f:Y\to X$ and $N$ is the negative part of $-f^{\ast}(K_{X}+\Delta)$. Since $a(F;X,\Delta)$ and $\mult_{F}$ only depend on the valuation $\ord_{F}$, potential log discrepancy also only depend on the valuation $\ord_{F}$.

\item The pair $(X,\Delta)$ is said to be \textit{potentially klt} if $\inf_F\bar{a}(F;X,\Delta)>0$ where $\inf$ is taken over all the prime divisors $F$ over $X$. For a real number $\epsilon\geq 0$, the pair $(X,\Delta)$ is said to be $\epsilon$-plc if $\inf_F\bar{a}(F;X,\Delta)\ge \epsilon$.
\end{enumerate}
\end{definition}

\begin{remark}
Unlike the usual klt or $\epsilon$-lc, to determine whether a given pair $(X,\Delta)$ with pseudoeffective $-(K_X+\Delta)$ is pklt or $\epsilon$-plc, we need to consider all the prime divisors over $X$ not just on a fixed model of $X$. However, if there exists a log resolution $f:Y\to X$ of $(X,\Delta)$ such that $-f^*(K_X+\Delta)$ admits the Zariski decomposition, then it is enough to consider the potential log discrepancies only for the divisors $F$ on such fixed $Y$ in order to determine whether the pair $(X,\Delta)$ is pklt or not (see Theorem \ref{thrm:Zariski resolution}).
\end{remark}

We need the following lemma to prove Theorem \ref{thrm:Zariski resolution}.

\begin{lemma}[\protect{\cite[Lemma 0-2-13]{KMM}}]\label{lem:KMM0-2-13}
Let $f:V\to W$ be a birational morphism of smooth projective varieties, $D$ an effective divisor on $V$, and $B$ a reduced simple normal crossing divisor on $W$. If $f^{-1}(B)\subseteq D$, then there exists an effective divisor $E$ on $V$ such that
$$ K_{V}+D=f^{\ast}(K_{W}+B)+E. $$
\end{lemma}

\begin{theorem}\label{thrm:Zariski resolution}
For a given pair $(X,\Delta)$ with pseudoeffective $-(K_X+\Delta)$, suppose that there exists a log resolution $f:Y\to X$ of $(X,\Delta)$ for which  $-f^*(K_X+\Delta)$ admits the Zariski decomposition $-f^*(K_X+\Delta)=P+N$. If $\overline{a}(G,X,\Delta)>0$ for all (finitely many) prime divisors $G\subseteq\Supp(f^{-1}_*\Delta)\cup\Exc(f)\cup\Supp N$, then the pair $(X,\Delta)$ is pklt.
\end{theorem}

\begin{proof}
We use the argument used in \cite[Lemma 0-2-12]{KMM}. Suppose that $f:Y\to X$ is a log resolution satisfying the given conditions.
We have $K_Y=f^*(K_X+\Delta)+\sum e_jF_j$ for some real numbers $e_j$ and prime divisors $F_j$ on $Y$.
Let $-f^*(K_X+\Delta)=P+N$ be the Zariski decomposition with $N=\sum \sigma_jF_j$.
Let $\delta=\min\{\overline{a}(G;X,\Delta)-1|G\subseteq\Supp(f^{-1}_*\Delta)\cup\Exc(f)\cup\Supp N\}$. Then $\delta=\min\{\overline{e}_j:=e_j-\sigma_j\}>-1$. Let $g:Z\to Y$ be a birational morphism which gives another log resolution $f\circ g:Z\to X$ of $(X,\Delta)$. Define by $B=\sum_{\bar{e}_{j}<0}F_{j}$ and $D=(g^{\ast}B)_{\text{red}}$. By Lemma \ref{lem:KMM0-2-13}, we have
$$ K_{Z}+D=g^{\ast}(K_{Y}+B)+E $$
for some effective divisor $E$. This implies that
\begin{align*}
K_{Z}&=g^{\ast}f^{\ast}(K_{X}+\Delta)+g^{\ast}\bigg(\sum_{j}e_{j}F_{j}\bigg)+g^{\ast}B+E-D\\
&=g^{\ast}f^{\ast}(K_{X}+\Delta)\\
&\phantom{=}\;\;+g^{\ast}\bigg(\sum_{\bar{e}_{j}<0}(1+e_{j})F_{j}\bigg)+g^{\ast}\bigg(\sum_{\bar{e}_{j}\ge 0}e_{j}F_{j}\bigg)+E-D.%\bigg(g^{\ast}\bigg(\sum_{\bar{a}_{j}<0}F_{j}\bigg)\bigg)_{\text{red}}.
\end{align*}
Note that we have the Zariski decomposition $-(f\circ g)^*(K_X+\Delta)=P'+N'$ where $N'=g^*(N)=g^*(\sum\sigma_i F_i)$. By definition, the potential log discrepancy $\overline{a}(G,X,\Delta)$ at a prime divisor $G$ on $Z$ is $1+$ the multiplicity of the following divisor along $G$:
%is given by the multiplicity along $G$ of the following divisor:
%\begin{align*}
%&g^{\ast}\bigg(\sum_{\bar{a}_{j}<0}(1+a_{j})F_{j}\bigg)+g^{\ast}\bigg(\sum_{\bar{a}_{j}\ge 0}a_{j}F_{j}\bigg)+E-\bigg(g^{\ast}\bigg(\sum_{\bar{a}_{j}<0}F_{j}\bigg)\bigg)_{\text{red}}-N'\\
%=&g^{\ast}\bigg(\sum_{\bar{a}_{j}<0}(1+\bar{a}_{j})F_{j}\bigg)+g^{\ast}\bigg(\sum_{\bar{a}_{j}\ge 0}\bar{a}_{j}F_{j}\bigg)+E-\bigg(g^{\ast}\bigg(\sum_{\bar{a}_{j}<0}F_{j}\bigg)\bigg)_{\text{red}}.
%\end{align*}
\begin{align*}
&g^{\ast}\bigg(\sum_{\bar{e}_{j}<0}(1+e_{j})F_{j}\bigg)+g^{\ast}\bigg(\sum_{\bar{e}_{j}\ge 0}e_{j}F_{j}\bigg)+E-D-N'\\
=&g^{\ast}\bigg(\sum_{\bar{e}_{j}<0}(1+\bar{e}_{j})F_{j}\bigg)+g^{\ast}\bigg(\sum_{\bar{e}_{j}\ge 0}\bar{e}_{j}F_{j}\bigg)+E-D \tag{$*$}\label{$*$}.
\end{align*}
Denote by $A_{1}, A_{2}$ the first two terms of (\ref{$*$}) above, respectively. Let $G$ be a prime component of $D=(g^*B)_{\text{red}}$. Then $\mult_GD=1$. Since $\bar{e}_{j}\ge \delta$ for all $j$, we have $\mult_GA_{1}\ge 1+\delta>0$. Obviously, $\mult_GA_{2} \ge 0$ and $\mult_GE \ge 0$. Therefore,
$$ \bar{a}(G;X,\Delta)=1+\mult_GA_{1}+\mult_GA_{2}+\mult_GE-\mult_GD\ge 1+\delta>0. $$
If $G$ is not a component of $D=(g^{\ast}B)_{\text{red}}$, then $\mult_GD=0$ and $\bar{a}(G;X,\Delta)\ge 1$.
\end{proof}

The following birational invariant property is the motivation of definition of potential pairs.

\begin{proposition}[cf. \protect{\cite[Proposition 3.11]{CP16}}]\label{prop:MMP invariant}
Let $(X,\Delta)$ be an $\epsilon$-plc pair. Suppose that there is a $-(K_{X}+\Delta)$-negative map $\varphi:X\dashrightarrow X'$. Then for any prime divisor $F$ over $X$, we have
$$\bar{a}(F;X,\Delta)=\bar{a}(F;X',\Delta'),$$
where $\Delta'$ is the birational transform of $\Delta$ on $X'$.
\end{proposition}

If we assume that $\varphi$ is the $-(K_X+\Delta)$-MMP, we have $\bar{a}(F;X',\Delta')=a(F;X',\Delta')$ since $-(K_{X'}+\Delta')$ is nef.
Thus the resulting model $(X',\Delta')$ is an $\epsilon$-lc pair.

\begin{definition}
A $\Q$-factorial normal projective variety $X$ is said to be of \emph{Fano type} if there exists a boundary $\Q$-divisor $\Delta$ on $X$ such that $(X,\Delta)$ is klt and $-(K_X+\Delta)$ is ample. %%See \cite{PS} for other characterizations.
\end{definition}

The following is a characterization of Fano type varieties in terms of potential pairs.

\begin{theorem}[\protect{\cite[Theorem 5.1]{CP16}}] \label{thrm:FT criterion}
Let $X$ be a $\Q$-factorial normal projective variety.
Then $X$ is a Fano type variety if and only if $-K_X$ is big and  $(X,\Delta)$ is a pklt pair for some divisor $\Delta$.
\end{theorem}

If $-(K_X+\Delta)$ is pseudoeffective and $D$ is a divisor on $X$, then we can define the pseudoeffective threshold as follows
$$
t_{\eff}(X,\Delta;D):=\sup\{\lambda|-(K_X+\Delta+\lambda D)\;\text{is pseudoeffective}\}.
$$

Suppose that $(X,\Delta)$ is a log canonical pair and $D$ is an effective $\mathbb{R}$-Cartier divisor on $X$. The \textit{log canonical threshold} of $D$ with respect to $(X,\Delta)$ is
$$
\lct(X,\Delta;D):=\sup\{t\in \R_{\geq0} ~|~ \text{$(X,\Delta+tD)$ is log canonical}\}.
$$

Note that the notion of pklt, plc can be defined only for the pairs $(X,\Delta+tD)$ with $0\leq t\leq t_{\eff}(X,\Delta;D)$.

\begin{definition}
Let $(X,\Delta)$ be a pair such that $-(K_X+\Delta)$ is pseudoeffective.
Assume that $(X,\Delta)$ is potentially klt.
For a divisor $D$ on $X$, suppose that $t_{\eff}=t_{\eff}(X,\Delta;D)>0$.
The \emph{plc threshold} $\plct(X,\Delta;D)$ of $(X,\Delta)$ with respect to $D$ is defined as
$$
\plct(X,\Delta;D):=\begin{cases}\sup\{\lambda\geq 0|(X,\Delta+\lambda D) \text{ is plc}\} & \text{if $(X,\Delta+t_{\eff}D)$ is not pklt}\\
t_{\eff}& \text{otherwise.}
\end{cases}
$$
If $t_{\eff}=0$, then we define $\plct(X,\Delta;D)=0$.
\end{definition}

By definition, $\plct(X,\Delta;D)\le t_{\eff}(X,\Delta;D)$. Note that if $t_{\eff}>0$ and $-(K_X+\Delta+tD)$ is nef for any $t\in[0,t_{\eff}]$, then $\plct(X,\Delta;D)=\text{lct}(X,\Delta;D)$. Furthermore, if $(X,\Delta+tD)$ is potentially klt for any $t\in[0,t_{\eff}]$, then $\plct(X,\Delta;D)= t_{\eff}(X,\Delta;D)$. Thus in some cases, the effective threshold $t_{\eff}(X,\Delta;D)$ and the log canonical threshold $\text{lct}(X,\Delta;D)$ appear as plc threshold $\plct(X,\Delta;D)$.

%Note that if $(X,\Delta)$ is a pklt pair such that $-(K_X+\Delta)$ is big, then $t_{\eff}(X,\Delta;D)>0$ and Proposition \ref{prop:plct=rclct} shows that $\plct(X,\Delta;D)$ always appears as the supremum $\sup\{\lambda|(X,\Delta+\lambda D)\;\text{is potentially lc}\}$.
%We also suspect that the same holds when $(X,\Delta)$ is plc and $t_{\eff}>0$.

%\begin{question}
%Suppose that $(X,\Delta)$ is a potentially lc pair and $-(K_X+\Delta)$ is pseudoeffective. Let $D$ be an effective divisor such that $t_{\eff}=t_{\eff}(X,\Delta;D)>0$. Then does the following inequality hold?
%$$
%\sup\{\lambda\in\R_{\geq 0}|(X,\Delta+\lambda D)\;\text{is potentially lc}\}\leq t_{\eff}
%$$
%\end{question}

%The properties of the log canonical threshold $\lct(X,\Delta,D)$ have been studied extensively recently. For example,.....
%However, not much is known about the threshold $t_{\eff}$ although this threshold will play an important role in the development of $(-K_X)$-MMP.
%As far as the $K_X$-minimal model program is concerned, the following variation of $t_{\eff}$ is more suitable than $t_{\eff}$:
%$$
%\tau(X,\Delta,D)=\inf\{\lambda\;|\;K_X+\Delta+\lambda D\;\text{is pseudoeffective}\}.
%$$
%By Cerbo \cite{}, it is knonw that

Our main result is a generalization of the following result.

\medskip

\begin{theorem}[\cite{HMX14}] \label{thrm:ACC}
Suppose that $n$ is a positive integer, $I$ is a subset of nonnegative real numbers satisfying DCC. Then there exists an ACC set $I'$ such that the log canonical threshold $\mathrm{lct}(X,\Delta;D)$ belongs to $I'$, whenever  the pair $(X,\Delta)$ and the divisor $D$ satisfy the followings:
\begin{enumerate}[label=\upshape(\roman*)]
\item $(X,\Delta)$ is a log canonical pair with $X$ of dimension $n$,
\item the coefficients of $\Delta$ and $D$ belong to $I$.
\end{enumerate}
\end{theorem}

\subsection{Finite generation}\hfill

Let $X$ be a normal projective variety. For an $\mathbb{R}$-divisor $D$ on $X$, we consider the group of global sections of $D$:
$$H^{0}(X,D)=\{f\in \mathbb{C}(X) ~|~ \mathrm{div}f+D\ge 0\}\cup \{0\}.$$
Let $D_{1},\cdots,D_{r}$ be given fixed $\mathbb{Q}$-divisors on $X$. Then we define the section ring associated to $D_{1},\cdots,D_{r}$ as
\[R=R(X;D_{1},\dots,D_{r}):=\bigoplus_{(m_{1},\cdots,m_{r})\in \mathbb{N}^{r}}H^{0}(X,m_{1}D_{1}+\cdots+m_{r}D_{r})\]
and the \textit{support} of $R$ is defined as
$$\Supp R=\left\{\left.D\in\sum\mathbb{R}_{\ge0}D_{i} ~\right|~ |D|_{\mathbb{R}}\neq \emptyset \text{ in } \Div_{\mathbb{R}}(X)\right\}.$$

\begin{theorem}[\cite{KKL16}]\label{thrm:geography}
Let $X$ be a normal projective variety and let $D_{1},\cdots, D_{r}$ be $\mathbb{Q}$-Cartier $\mathbb{Q}$-divisors on $X$. Assume that the section ring $R=R(X;D_{1},\dots,D_{r})$ is finitely generated. Then the followings hold:
\begin{enumerate}[label=\upshape(\arabic*)]
\item $\Supp R$ is a rational polyhedral cone.
\item Suppose that $\Supp R$ contains a big divisor. If $D\in \sum\mathbb{R}_{\ge0}D_{i}$ is pseudoeffective, then $D$ is in $\Supp R$.
\item There is a finite decomposition $\Supp R=\bigcup \mathcal{C}_{i}$ such that each $\mathcal{C}_{i}$ is a rational polyhedral cone, $\sigma_{\Gamma}$ is linear on $\mathcal{C}_{i}$ for every prime divisor $\Gamma$ over $X$, and the cones $\mathcal{C}_{i}$ form a fan.
\item There is a positive integer $d$ and a resolution $f:Y\to X$ such that $\Mob f^{\ast}(dD)$ is base point free for every $D\in \Supp R\cap \Div(X)$, and $\Mob f^{\ast}(kdD)=k\Mob f^{\ast}(dD)$ for every positive integer $k$.
\end{enumerate}
\end{theorem}

\section{Proofs of Theorems}

We first prove the ACC property of pseudoeffective thresholds $t_{\eff}(X,\Delta;D)$.

\begin{proof}[Proof of Theorem \ref{thrm:main t_eff}]
Since $(X,\Delta)$ is $\epsilon$-plc and $-(K_{X}+\Delta)$ is big, $X$ is of Fano type by Theorem \ref{thrm:FT criterion}. Therefore the section ring $R=R(X;-(K_{X}+\Delta),-(K_{X}+\Delta+t_{\eff}D))$ is finitely generated. For a log resolution $f:Y\to X$ as in Theorem  \ref{thrm:geography} (4), let $-f^{\ast}(K_{X}+\Delta+tD)=P(t)+N(t)$ be the good Zariski decomposition. Then $P(t)$ is semiample for all $0\le t\le t_{\eff}=t_{\eff}(X,\Delta;D)$. Note that $t_{\eff}>0$ because $-(K_{X}+\Delta)$ is big. Take a sufficiently small $\delta>0$ and let $g:Y\to X'$ be the semiample fibration associated to $P(t_{\eff}-\delta)$. Then by \cite[Theorem 4.2]{KKL16}, the induced birational map $\varphi:X\dashrightarrow X'$ is the ample model of $-(K_{X}+\Delta+tD)$ for all $t_{\eff}-\delta\le t<t_{\eff}$, and at the same time, it is also a semiample model of $-(K_{X}+\Delta+t_{\eff}D)$.  Let $h:X'\to Z'$ be the semiample fibration associated to $-(K_{X'}+\Delta'+t_{\eff}D')=-\varphi_{\ast}(K_{X}+\Delta+t_{\eff}D)$ and let $F'$ be a general fiber of $h$. Note that $\dim F'>0$ because $-(K_{X'}+\Delta'+t_{\eff}D')$ is not big.

Now we claim that $F'$ belongs to some bounded family.
%Let $-f^*(K_X+\Delta)=P+N$ be the good Zariski decomposition. Since $g\circ f^{-1}:X\dashrightarrow X'$ is $-(K_X+\Delta)$-negative, we have the good Zariski decomposition $-g^*(K_Y+\Delta_Y)=P+N'$ where $N'\leq N$.
Let $K_{Y}+\Delta_{Y}=f^{\ast}(K_{X}+\Delta)$, $K_{Y}+\Delta_{Y}+E_{Y}=g^{\ast}(K_{X'}+\Delta')$ and $P+N=-f^{\ast}(K_{X}+\Delta)$ the good Zariski decomposition. Then we can write $-P-N+E_{Y}=g^{\ast}(K_{X'}+\Delta')$. Since $D$ is nef, $N\le N(t_{\eff})$. By \cite[Lemma 4.1]{KKL16}, $N(t_{\eff})$ is $g$-exceptional and so is $N$. Moreover, by definition, $E_{Y}$ is $g$-exceptional. Thus, by the negativity lemma, $E_{Y}\le N$.

By assumption (i), we know that $(X',\Delta')$ is $\epsilon$-lc. Thus the pair $(F',\Delta'|_{F'})$ is also $\epsilon$-lc. Now consider the restriction
\[-(K_{X'}+\Delta'+t_{\eff}D')|_{F'}=-(K_{F'}+\Delta'|_{F'}+t_{\eff}D'|_{F'})\equiv 0.\]
Since \[\delta D'|_{F'}\equiv -(K_{F'}+\Delta'|_{F'}+(t_{\eff}-\delta)D'|_{F'})=-(K_{X'}+\Delta'+(t_{\eff}-\delta)D')|_{F'}\]
and the right-hand side is ample, $D'|_{F'}$ is ample. This implies that $-(K_{F'}+\Delta'|_{F'})$ is ample. By \cite[Theorem 1.1]{Bir21}, $F'$ belongs to a bounded family $\mathcal{F}$.

Finally by arguing as in the proof of \cite[Theorem 1.3]{HL17}, we finish the proof as follows. Note first that we can write for some very ample Cartier divisor $M_{F'}$ on $F'$
$$
t_{\eff}=\frac{1}{d}\left(-K_{F'}\cdot M_{F'}^{\dim F'-1}-\Delta'|_{F'}\cdot M_{F'}^{\dim F'-1}\right)
$$
where $d=D'|_{F'}\cdot M_{F'}^{\dim F'-1}$.
By the boundedness of $F'$, we may assume that $-K_{F'}\cdot M_{F'}^{\dim F'-1}\le r$ for some fixed number $r>0$. It is easy to see that $d=D'|_{F'}\cdot M_{F'}^{\dim F'-1}$ belongs to a DCC set (e.g., $I':=\{a\cdot m|a\in I, m\in \mathbb N\}$). Note also that $\Delta'|_{F'}\cdot M_{F'}^{\dim F'-1}$ belongs to the DCC set $I'$. Therefore, it is immediate that $t_{\eff}$ belongs to some ACC set.
\end{proof}

Due to the following example, the $\epsilon$-plc condition on $(X,\Delta)$ with $\epsilon>0$ in Theorem \ref{thrm:main t_eff} cannot be strengthened to plc condition (i.e., $\epsilon=0$).

\begin{example}\label{eg:e=0}
Let $X_{n}=\mathbb{P}_{\mathbb{P}^{1}}(\mathcal{O}\oplus \mathcal{O}(n))$ be a Hirzebruch surface, $M$ the negative section, and $F$ a $\mathbb{P}^{1}$-fiber. Then $-K_{X_{n}}=((1+\frac{2}{n})M+(n+2)F)+(1-\frac{2}{n})M$ is the Zariski decomposition with the negative part $N=(1-\frac{2}{n})M$. By Theorem \ref{thrm:Zariski resolution}, we know that $X_{n}$ is pklt. If we let $D_n=M+(n+3)F$, then one can see that $t_{\eff}(X_{n},0,D_n)=\frac{n+2}{n+3}$. Clearly, $\left\{\frac{n+2}{n+3}\left|n\in\mathbb N\right.\right\}$ does not satisfy ACC. This shows that the condition being $\epsilon$-plc in Theorem \ref{thrm:main t_eff} is necessary. Note that $\bar{a}(M;X_{n},0)=\frac{2}{n}$.
\end{example}

The following result is used in the proof of Theorem \ref{thrm:main}.

\begin{theorem}[\protect{\cite[Theorem 1.5]{HMX14}}] \label{thrm:ACC2}
Fix a positive integer $n\in \mathbb{N}$ and a DCC set $I\subseteq [0,1]$. Then there is a finite subset $I_{0}\subseteq I$ such that if
\begin{enumerate}[label=\upshape(\arabic*)]
\item $X$ is a projective variety of dimension $n$,
\item $(X,\Delta)$ is log canonical,
\item $\Delta=\sum \delta_{i}\Delta_{i}$,  where $\delta_{i}\in I$,
\item $K_{X}+\Delta\equiv 0$,
\end{enumerate}
then $\delta_{i}\in I_{0}$.
\end{theorem}

Now we prove our main result Theorem \ref{thrm:main}.

\begin{proof}[Proof of Theorem \ref{thrm:main}]
Note that by definition, we have $\plct(X,\Delta;D)\leq t_{\eff}(X,\Delta;D)$. We will prove the ACC of the $\plct(X,\Delta;D)$ for the following two cases separately; $\Sigma=\Sigma_1\cup\Sigma_2$ where
$$
\begin{array}{l}
\Sigma_1:=\{\lambda_1|\lambda_1=\plct(X,\Delta;D)<t_{\eff}(X,\Delta;D)\}, \\
\Sigma_2:=\{\lambda_2|\lambda_2=\plct(X,\Delta;D)=t_{\eff}(X,\Delta;D)\}.
\end{array}
$$

Suppose first that $\lambda_{1}\in\Sigma_1$, that is, $\lambda_{1}=\plct(X,\Delta;D)<t_{\eff}(X,\Delta;D)=t_{\eff}$. Since $X$ is of Fano type by Theorem \ref{thrm:FT criterion}, the ring $R(X,-(K_{X}+\Delta), -(K_{X}+\Delta+\lambda_{1}D))$ is finitely generated. Take a resolution $f:Y\to X$  as in Theorem \ref{thrm:geography} (4). Let $-f^{\ast}(K_{X}+\Delta+\lambda D)=P(\lambda)+N(\lambda)$ be the good Zariski decomposition for $0\leq \lambda\leq \lambda_{1}$. There exists a semiample fibration $g:Y\to X'$ associated to all $P(\lambda_{1}-\delta)$ with sufficiently small $\delta>0$. By \cite[Theorem 4.2]{KKL16}, the induced rational map $\varphi:X\dashrightarrow X'$ is a semiample model for $-(K_{X}+\Delta+\lambda D)$ for any $\lambda\in[\lambda_{1}-\delta,\lambda_{1}]$. Let $K_{X'}+\Delta'+\lambda D'=\varphi_{\ast}(K_{X}+\Delta+\lambda D)$. Then $K_{Y}+\Delta_{Y}+\lambda f^{\ast}D+N(\lambda)=g^{\ast}(K_{X'}+\Delta'+\lambda D')$. By Remark \ref{remark_pklt}, $(X,\Delta+\lambda D)$ is pklt if and only if the coefficients of $\Delta_{Y}+\lambda f^{\ast}D+N(\lambda)$ are less than 1. Therefore we obtain that $\plct(X,\Delta;D)=\lct(X',\Delta';D')$. Hence by Theorem \ref{thrm:ACC}, the set $\Sigma_{1}=\{\lambda_1 ~|~ \lambda_1=\plct(X,\Delta;D)<t_{\eff}(X,\Delta;D)\}$ satisfies ACC.

Now, consider the set $\Sigma_{2}=\{\lambda_2 ~|~ \lambda_{2}=\plct(X,\Delta;D)=t_{\eff}(X,\Delta;D)\}$. We will show that $\Sigma_{2}$ also satisfies ACC. For $\lambda_2\in \Sigma_{2}$, take $X, X',\Delta,\Delta', D$ and $D'$ as above. Then, $\lambda_2=\plct(X,\Delta;D)<\lct(X',\Delta';D')$, because $(X,\Delta+\lambda_{2}D)$ is pklt. Therefore we cannot apply Theorem \ref{thrm:ACC}. Instead, we proceed using the following argument.

Let $t_{1}\le t_{2}\le \cdots$ be a nondecreasing sequence of $\Sigma_{2}$. Then for each $i$, there are a variety $X'_{i}$ and divisors $\Delta'_{i},D'_{i}$ associated to $t_{i}$ as we constructed before. Let $h_{i}:X'_{i}\to Z'_{i}$ be the semiample fibration associated to $-(K_{X'_{i}}+\Delta'_{i}+t_{i}D'_{i})$ and $F'_{i}$ a general fiber of $h_{i}$. Then $(F'_{i},\Delta'_{i}|_{F'_{i}}+t_{i}D'_{i}|_{F'_{i}})$ is klt and
$$ K_{F'_{i}}+\Delta'_{i}|_{F'_{i}}+t_{i}D'_{i}|_{F'_{i}}\equiv 0. $$
By passing to a subsequence, we may assume that $\dim F'_{i}$ is constant. Since $\{t_{i}\}$ is nondecreasing, the coefficients of $\Delta'_{i}|_{F'_{i}}+t_{i}D'_{i}|_{F'_{i}}$ belong to a DCC set. By Theorem \ref{thrm:ACC2}, the coefficients of $\Delta'_{i}|_{F'_{i}}+t_{i}D'_{i}|_{F'_{i}}$, indeed, belong to a finite set. It follows that a sequence $\{t_{i}\}$ stabilizes.

Since $\Sigma=\Sigma_{1}\cup \Sigma_{2}$ and union of two ACC sets again satisfies ACC, as we desired, $\Sigma$ satisfies ACC.
\end{proof}

Lastly, we state the following relation to the threshold concerning the $\R$-complement of Shokurov.
We recall from \cite{sho} that a pair $(X,\Delta)$ is said to have an \emph{$\R$-complement} if there exists an effective divisor $\Delta'$ such that $\Delta'\ge \Delta$, $(X,\Delta')$ is lc, and $K_{X}+\Delta'\sim_{\R}0$.

\begin{proposition}\label{prop:plct=rclct}
Suppose that $(X,\Delta)$ is a pklt pair such that $-(K_{X}+\Delta)$ is big.
Then
$$
\plct(X,\Delta;D)=\sup\{t\geq 0\;|\;(X,\Delta+tD)\;\text{ has an $\R$-complement}\}.
$$
%$$\sup\{\lambda|(X,\Delta+\lambda D) \text{ is plc}\}\leq  t_{\eff}(X,\Delta;D).$$
\end{proposition}

The threshold on the righthand side is called the \emph{$\R$-complement threshold} in  \cite{sho} and \emph{$\R$-complementary threshold} in \cite{HLS}.

\begin{proof}
Note that by Theorem \ref{thrm:FT criterion} the variety $X$ in the given pair $(X,\Delta)$ is a Fano type variety.
We claim that if $(X,\Delta)$ is an lc pair with a Fano type variety $X$, $(X,\Delta)$ has an $\R$-complement if and only if $-(K_{X}+\Delta)$ is pseudoeffective and $(X,\Delta)$ is plc.

Let $\Delta'$ be an $\R$-divisor on $X$ such that $\Delta'\ge \Delta$, $(X,\Delta')$ is lc, and $K_{X}+\Delta'\sim_{\R}0$. Then, clearly, $-(K_{X}+\Delta)$ is pseudo-effective. Furthermore, since $-(K_{X}+\Delta')$ is $\R$-trivial, $(X,\Delta')$ is plc and hence $(X,\Delta)$ is also plc. Conversely, assume that $-(K_{X}+\Delta)$ is pseudoeffective and $(X,\Delta)$ is plc. Since $X$ is of Fano type, we can find a resolution $f:Y\to X$ such that $-f^{\ast}(K_{X}+\Delta)$ admits the good Zariski decomposition $-f^{\ast}(K_{X}+\Delta)=P+N$ with the semiample positive part $P$. Pick a general element $P'\in |P|$ and let $\Delta':=\Delta+f_{\ast}P'+f_{\ast}N$. Then it is easy to check that $\Delta'$ is an $\R$-complement of $(X,\Delta)$. By \cite[Theorem 6]{sho}, we have
$$
\plct(X,\Delta;D)=\sup\{t\geq 0\;|\;(X,\Delta+tD)\;\text{ has an $\R$-complement}\}.
$$
%Note that if $(X,\Delta)$ has an $\R$-complement, then $-(K_{X}+\Delta)$ is pseudoeffective. Therefore, we have
%$$
%\sup\{t\in \R_{\ge0} ~|~ \text{$(X,\Delta+tD)$ has an $\R$-complement}\}\leq t_{\eff}(X,\Delta;D).
%$$
\end{proof}

Now we see that Theorem \ref{thrm:main} is also implied by  \cite[Theorem 21]{sho} and \cite[Theorem 8.20]{HLS}.

%\begin{remark}
%In Proposition 4.4 of \cite{CP16}, it is shown that for the pair $(X,\Delta)$ in Proposition \ref{prop:plct=rclct}, there exists an $\R$-complement $\Delta'\geq \Delta$ such that $(X,\Delta')$ is klt.
%\end{remark}

\section{Boundedness of $\epsilon$-plc pairs}\label{sec4}

It is known that the ACC of log canonical thresholds (Theorem \ref{thrm:ACC}) is implied by the following Birkar's result on the boundedness of $\epsilon$-lc Fano varieties (\cite{HMX14}).

\begin{theorem}[{\cite[Theorem 1.1]{Bir21}}]\label{thrm:BAB}
Let $d$ be a natural number and $\epsilon$ a positive real number. Then the projective varieties $X$ such that
\begin{enumerate}[label=$\bullet$]
\item $(X,\Delta)$ is $\epsilon$-lc of dimension $d$ for some boundary $\Delta$, and
\item $-(K_{X}+\Delta)$ is nef and big,
\end{enumerate}
form a bounded family.
\end{theorem}

%\comment{In the proof, he showed that if there is a $-K_{X}$-negative map $X\dashrightarrow X'$, then the boundedness of $X'$ implies that the boundedness of $X$. Immediately, it follows that the family of $\epsilon$-plc varieties with big anticanonical divisor is bounded.}

Let $d$ be a positive integer and $\epsilon$ a nonnegative real number. Let us consider the following collections of Fano type varieties:

$$
\begin{array}{rl}
\FT(d,\epsilon)&:=\left\{X\left| \begin{array}{l}
                      \text{$X$ is a $\Q$-factorial normal projective variety of dimension $d$,} \\
                      \text{$(X,\Delta)$ is $\epsilon$-lc and $-(K_X+\Delta)$ is nef and big for some boundary $\Delta$}
                    \end{array}\right.\right\},\\
                    &\\
\pFT(d,\epsilon)&:=\left\{X\left| \begin{array}{l}
                      \text{$X$ is a $\Q$-factorial normal projective variety of dimension $d$,} \\
                      \text{$(X,\Delta)$ is $\epsilon$-plc and $-(K_X+\Delta)$ is big for some boundary $\Delta$}
                    \end{array}\right.\right\}.
\end{array}
$$
We will write $\FT(d,0)=\FT(d)$ and $\pFT(d,0)=\pFT(d)$.

\bigskip

By definition, $\FT(d,\epsilon)\subseteq\pFT(d,\epsilon)$ for any $\epsilon\geq 0$ and the following holds: $$\FT(d)=\bigcup_{\epsilon>0}\FT(d,\epsilon)=\bigcup_{\epsilon>0}\pFT(d,\epsilon).$$
Theorem \ref{thrm:BAB} proves that $\FT(d,\epsilon)$ is bounded for any $\epsilon>0$.
It is well known that the Fano type varieties are Mori dream spaces by \cite{bchm}.
Thus the varieties $X\in\FT(d,\epsilon)$ can be understood as the images of some $X'\in\pFT(d,\epsilon)$ that are obtained by running the Minimal model program on some $\epsilon$-lc pair $(X',\Delta)$.
% of an $\epsilon$-plc pair $(X',\Delta)$ such that $-(K_{X'}+\Delta)$ is big and $X'\in\pFT(n,\epsilon)$.

\begin{theorem}
Let $d$ be a positive integer and $\epsilon$ a positive real number. Then $\pFT(d,\epsilon)$ is a bounded family.
\end{theorem}

\begin{proof}
Let $X\in\pFT(d,\epsilon)$. Then $X$ is of Fano type and we can run $-K_{X}$-MMP. Let $\phi:X\dashrightarrow X'$ be a $-K_{X}$-MMP which terminates with nef $\mathbb{Q}$-divisor $-K_{X'}$. Then $X'$ is $\epsilon$-lc and $-K_{X'}$ is nef and big. Thus, by \cite[Theorem 1.1]{Bir21}, $X'$ belongs to a bounded family. This implies that there is a positive integer $n$ such that $-nK_{X'}$ is Cartier. Now by the effective base point free theorem, we can also assume that $|-nK_{X'}|$ is base point free. Let $H'\in |-nK_{X'}|$ be a general element and let $B'=\frac{1}{n}H'$. Then $K_{X'}+B'$ is a klt $n$-complement of $K_{X'}$ and it gives a klt $n$-complement $K_{X}+B$ of $K_{X}$ (cf. \cite[6.1.(3)]{Bir19}). Finally, by \cite[Theorem 1.3]{HX15}, $X$ belongs to a bounded family.
\end{proof}

We also have the following result.
\begin{proposition}\label{prop:FT<pFT}
For any $\epsilon>0$, there exists $\delta>0$ such that $\FT(d,\epsilon)\subsetneq\pFT(d,\delta)$.
\end{proposition}
\begin{proof}
We know that for a fixed $\epsilon>0$, $\FT(d,\epsilon)\subseteq\pFT(d,\epsilon)\subseteq\pFT(d,\epsilon')$ for any $\epsilon'<\epsilon$. Suppose that $\FT(d,\epsilon)=\pFT(d,\epsilon')$ for any $0<\epsilon'<\epsilon$.
Note that
$$
\FT(d)=\bigcup_{r>0}\pFT(d,r)=\bigcup_{\epsilon>r>0}\pFT(d,r)=\FT(d,\epsilon).
$$
However, it is clear that $\FT(d)\neq\FT(d,\epsilon)$.
\end{proof}

Proposition \ref{prop:FT<pFT} implies that $\pFT(d,\epsilon)$ is not necessarily of the form $\FT(d,\epsilon)$. Thus Theorem \ref{thrm:BAB} cannot be applied directly to obtain the boundedness of $\pFT(d,\epsilon)$.
The following proposition gives an alternative proof for the boundedness of $\pFT(d,\epsilon)$.

\begin{proposition}\label{prop:pFT leq FT}
Let $\epsilon$ be a positive real number and $d$ a positive integer. Then $\pFT(d,\epsilon)\subseteq \FT(d,\epsilon')$ for any $0<\epsilon'<\epsilon$.
\end{proposition}

\begin{proof}
Suppose that $X\in\pFT(d,\epsilon)$. Then there exists a boundary divisor $\Delta$ on $X$ such that $-(K_{X}+\Delta)$ is big and $(X,\Delta)$ is $\epsilon$-plc. If $-(K_{X}+\Delta)$ is nef, then obviously, $X\in \FT(d,\epsilon)$. Thus we assume that $-(K_{X}+\Delta)$ is not nef. Since $X$ is of Fano type, the section ring $R(-(K_{X}+\Delta))$ is finitely generated. Let $f:Y\to X$ be a log resolution of $(X,\Delta)$ and let $-f^{\ast}(K_{X}+\Delta)=P+N$ be the good Zariski decomposition. Write $K_{Y}+\Delta_{Y}=f^{\ast}(K_{X}+\Delta)$. Take a general element $P_{0}\in |P|_{\Q}$ and let $D:=f_{\ast}(P_{0}+N)$. Then $K_{X}+\Delta+D\sim_{\Q}0$ and $f^{\ast}(K_{X}+\Delta+D)=K_{Y}+\Delta_{Y}+N+P_{0}$. Since $-(K_{X}+\Delta)$ is big, there are ample $\Q$-divisor $A$ and an effective $\Q$-Cartier $\Q$-divisor $\Delta'$ such that $-(K_{X}+\Delta)\sim_{\Q}A+\Delta'$. Let $\delta$ be a positive real number. Consider the pair $(X,\Delta+\delta\Delta'+(1-\delta)D)$. By construction,
\[f^{\ast}(K_{X}+\Delta+\delta\Delta'+(1-\delta)D)=K_{Y}+\Delta_{Y}+N+\delta(f^{\ast}\Delta'-N)+P_{0}.\]
If the coefficients of $f^{\ast}\Delta'$ are bounded, then we can show that the assertion is true by taking sufficiently small $\delta$. Since $-(K_{X}+\Delta)$ is not nef, by the cone theorem (cf. \cite[Theorem 3.7]{KM}), there exists a rational curve $C$ on $X$ such that $0<-C\cdot (K_{X}+\Delta)\le 2d$. Let $C'$ be a curve on $Y$ such that $f(C')=C$. Then we have
\[C'\cdot f^{\ast}\Delta'<C'\cdot f^{\ast}(A+\Delta')\le 2d.\]
Therefore, for any $y\in Y$,
\[\mult_{y}C'\cdot \mult_{y}f^{\ast}\Delta'<2d.\]
Hence the coefficients of $f^{\ast}\Delta'$ is bounded.
\end{proof}

Since $\FT(d,\epsilon')$ is bounded by Theorem \ref{thrm:BAB}, so is $\pFT(d,\epsilon)$.

%%%%%%%%%%%%%%%%%%%%%%%%%%%%%%%%%%%%%%%%%%%%%%%%%%%%%

\end{document}